\documentclass[reqno]{amsart}
\usepackage{graphics,verbatim,color,amsthm}
\usepackage{graphicx}
\usepackage[all]{xy}

\theoremstyle{plain}
\newtheorem{theorem}{Theorem}
\newtheorem{proposition}{Proposition}
\newtheorem{lemma}{Lemma}  
 
\newtheorem*{conjecture}{Conjecture}
\newtheorem{definition}{Definition}

\theoremstyle{remark}

\theoremstyle{definition}

\def\R{\mathbb{R}}
\def\Z{\mathbb{Z}}

\def\E{\mathbb{E}}
\def\L{\mathbb{L}}

\def\M{\mathbb{M}}

\def\H{\mathbb{H}}

\def\l{\lambda}

\def\smallskip{\par\vspace{1mm}}
\def\medskip{\par\vspace{2mm}}
\def\bigskip{\par\vspace{3mm}}

\def\sR{subRiemannian } 
\def\Ri{Riemannian } 

\def\fr#1#2{\frac{#1}{#2}}

\newcount\equationcount
\def\thenumber{0}
\def\eq#1{\global\advance\equationcount by 1
   \def\thenumber{\number\equationcount}
                        {$$#1\eqno(\thenumber)$$}}

\def\on{ orthonormal }

\newcommand{\dd}[2]
{
{{\partial #1}   \over {\partial #2}}
}

\begin{document}

\title[]{Keplerian Dynamics on the Heisenberg Group and Elsewhere}
 
\author{Richard Montgomery and Corey Shanbrom}
 
\address{Dept. of Mathematics\\ University of California, Santa Cruz\\ Santa Cruz CA 95064}

\email{rmont@count.ucsc.edu}

\address{Dept. of Mathematics\\ University of California, Santa Cruz\\ Santa Cruz CA 95064}

\email{cshanbro@ucsc.edu}

\date{December 11, 2012}



\thanks{For Jerry}

\begin{abstract}  Posing Kepler's problem of motion around a fixed ``sun'' requires the geometric mechanician to choose a metric and a Laplacian. The metric provides the kinetic energy. The fundamental solution to the Laplacian (with delta source at the ``sun'') provides the potential energy.  Posing Kepler's three laws (with input from Galileo) requires symmetry conditions.  The metric space must be homogeneous, isotropic, and admit dilations.  Any Riemannian manifold enjoying these three symmetry properties is Euclidean.  So if we want a semblance of Kepler's three laws to hold but also want to leave the Euclidean realm, we are forced out of the realm of Riemannian geometries.  The Heisenberg group (a subRiemannian geometry) and lattices provide the simplest examples of metric spaces enjoying a semblance of all three of the Keplerian symmetries.  We report success in posing, and solving, the Kepler problem on the Heisenberg group.  We report failures in posing the Kepler problem on the rank two lattice and partial success in solving the problem on the integers.  We pose a number of questions.  
 \end{abstract}

\maketitle


\section{Introduction}

Newton formulated and solved what we   call today ``Kepler's problem''
-- the problem whose negative energy solutions are  Keplerian ellipses. 
The essential backdrop to the problem is Euclidean three-space and  its
group of isometries and scalings.  Can we pose   Kepler's
problem on an arbitrary metric space?    What properties must
the space have if Kepler's three laws, or ghosts of these laws,  are to hold?

In `Foundations of Mechanics' (\cite{Foundations}), Abraham and Marsden formulate classical mechanics
as dynamical systems on the tangent bundle or cotangent bundle of a Riemannian manifold, which they
call  ``natural mechanical systems''.
 The \Ri metric defines the kinetic energy. One   must choose a potential energy.   
 In order to formulate   Kepler's problem   on our manifold,  we take this potential
 to be the fundamental solution to the Laplacian.  We choose a point on the manifold to be our ``sun," which is the delta function source of the fundamental solution.
Following Galileo we assume that the choice of location of the sun does not matter:
that is, we will assume that our space is homogeneous. 

The mildest \Ri departures from Euclidean space are the  spaces of constant curvature:
the sphere and hyperbolic space.   About a century and a half  before `Foundations,'
  Lobachevksy (\cite{Lobachevsky}), one of the founders of hyperbolic geometry,
 posed the Kepler problem as a ``natural mechanical system''  on hyperbolic space. 
Later, Serret  (\cite{Serret}) posed the Kepler problem on the sphere.  We are grateful to
F. Diacu for these references and his survey \cite{Diacu1}.
Kepler's 1st and 2nd laws hold in  each of the three   constant curvature geometries:
hyperbolic space, the 3-sphere, and the original flat Euclidean space.
But Kepler's  3rd law 
fails  for these  non-flat geometries  for the simple reason that they  
  admit no continuous  scaling symmetries, or ``dilations.''

We   argue that in order to even formulate Kepler's third law
our metric space must admit dilations.  But if a space  admits dilations, and is not Euclidean, then it  
cannot  even  be  Riemannian!  (We sketch the proof of this fact below.) 
The non-Euclidean spaces which admit  dilations are subRiemannian: they are the   Carnot groups.
The simplest Carnot group is  the   Heisenberg group.   

This observation brings us to our main problem: pose and solve the  
the Kepler problem on the Heisenberg group.
We will pose it.  We will not fully solve it.
We will  show that all periodic solutions to the Kepler problem on the Heisenberg group
must lie on the zero energy surface, and that the problem is integrable
{\it when restricted to this zero energy surface}.  Such solutions are described in Figures \ref{orbits} and \ref{orbit1}.  A modified version of Kepler's third law holds for the periodic solutions.

To  write down the  Kepler-Heisenberg problem we must
have an explicit expression for the  potential: it is  the  fundamental solution for the {\it sub}Laplacian on the
Heisenberg group. 
Luckily,  Folland  (\cite{Folland}) found such an expression    in 1970. 

We will also attempt to pose and solve Kepler's problem on some lattices.
Lattices  almost admit dilations: we can scale by  positive integer scaling factors,
but we cannot invert these scaling factors. 
 The  integers form the simplest  lattice.  
 We will pose and solve a Kepler problem on the integers.
 Our `solutions' are of a high school nature.  (We apologize in advance
 if our treatment here   embarrasses   readers with any skill 
 in numerical methods and discretization.)   These solutions  are  indicated in Figure \ref{intpic}.
 
 We then try to pose and solve the  Kepler problem on the  integer lattice
in the plane where   we run into fundamental problems which  
  lead us to believe that the very definition of   a  discrete dynamical  system
 is not yet well formulated.   The heart of this problem is that
 the differences of  values of the lattice potential -- that being the 
 fundamental solution of the lattice Laplacian -- are irrational.
  
{\sc Dedication and Acknowledgements.}
This article is dedicated to the memory of Jerry
and in thanks for all his inspiration.
We would also like to thank the GMC group,
in particular, David Mart\'{i}n de Diego, Juan Carlos Marrero,
and Edith Padron for inviting us to the summer school in 2011
outside of Madrid.
The formulation of the Kepler-Heisenberg  problem was inspired by talking with many of the participants at that summer school.

\section{Kepler's 3 laws in a metric space.} 
Let's recall    Kepler's three laws for the motion of planets around the sun.  

K1. Planets  travel on   ellipses with one focus the   sun. 

K2. Equal areas are swept out in equal times.  This law is equivalent to the conservation
of the planet's angular momentum about the sun.

K3.  Period-Size:  The   period $T$  of an orbit  and its  size $a$ (semi-major axis) are related by a 
universal monomial relation $a^3 =  C T^2$.   

The Keplerian planet  moves in a Euclidean space.
 Do  Kepler's    laws  even  make sense on  a  general metric space? 
If not, what restrictions must we impose on the metric space in order to make sense of a particular law?
We discuss what is required  of our metric space in order to formulate the corresponding law.

{\bf  K1.}  We can define an ``ellipsoid''  for  any metric space $X$. Fix two foci  $S, F \in X$  and    a 
positive number
$2a$.  Consider the locus of points $x \in X$ for which $d(S, x) + d(x, F) = 2a$.
If  this locus is  to be a curve then the  metric space must be two-dimensional,  
for example, a  smooth surface.   K1 requires then that $X$ is a two-dimensional, 
or that its Keplerian dynamics can be reduced to two-dimensions. 
 Kepler's problem has been posed and solved  
  satisfactorily  on the two-sphere and on the hyperbolic plane as described in the introduction.  Its   solutions satisfy K1.  

{\bf K2 }  is equivalent to conservation of angular momentum.  
Angular momentum is conserved if the   kinetic and potential energies in
Newton's formulation of the Kepler problem are invariant under rotations about the sun.
This requires isotropicness: all directions in the metric space  are the same, at least through
the sun.       
 The two-sphere and the hyperbolic plane enjoy rotational  symmetry and hence K2.

{\bf K3 }  is a  scaling law.  It is an immediate 
consequence of  the fact that the   Newtonian  potential $V(x) = c/|x|$
is homogeneous of degree $-1$.   This homogeneity implies the
space-time symmetry $x, t \mapsto \lambda x,  \lambda ^{3/2} t$, which is to say:
if $x(t)$ solves Kepler's equation then so does $\lambda x (\lambda ^{-3/2} t)$.
From one periodic solution $x(t)$  we generate a one-parameter    family $x_{\lambda}$.
The energy-period relation in K3 for  this family follows from the scaling symmetry

\subsubsection{Kepler's third law for homogeneous potentials in Euclidean space.} \label{thirdlaw}

Any homogeneous  potential $V$ on a Euclidean space enjoys a
 version of K3.  Homogeneity is a scaling symmetry:
 $x \mapsto \lambda x \implies V(x) \mapsto V(\lambda x) = \lambda ^{-\alpha} V(x)$.
 We try to extend the symmetry to time and velocities by a power law ansatz:
 $(x, t, v) \mapsto (\lambda x, \lambda ^{\beta} t,  \lambda ^{-\nu} v)$.
 Balancing   the resulting  scalings of potential and kinetic energies
 implies that $\nu =  \alpha/2$.  The requirement $v = dx/dt$ yields $\beta = 1+ (\alpha/2)$.  We are led to 
    the extended    scalings 
 $$(x, t,v ) \mapsto (\lambda x, \lambda ^{1 + \alpha/2} t ,  \lambda ^{- \alpha/2} v ).$$
In terms of   curves $\gamma(t)$, which are maps from $t$ to $x$-space,  the  scaling    operation
 is 
 $$\gamma (t) \mapsto \gamma_{\lambda} (t) = \lambda \gamma (\lambda^{- (1 + \alpha/2)} t).$$
 One verifies that if $\gamma$ satisfies Newton's equation $\ddot \gamma = - \nabla V( \gamma)$
 then so does $\gamma_{\lambda}$. (Use that $\nabla V$ is homogeneous of degree $-\alpha -1$.)
  The scaling symmetry thus  takes solutions of energy $H$ to solutions of energy $\lambda^{-\alpha} H$. 
  Now if $\gamma$ is periodic of period $T$ and with typical size $a$, then $\gamma_{\lambda}$ is periodic of
  period $\lambda ^{\beta} T = \lambda ^{1 + \alpha/2} T$ and typical size $\lambda a$. We thus arrive at our modified  K3:
  $T^2 = C a^{2 + \alpha}$.

\subsubsection{Dilations} To have an  analogue of K3,   our metric space must, like Euclidean space,    admit dilations.
\begin{definition} A dilation of the metric space $(X, d)$  with center $S \in X$ and  dilation factor $\lambda$  is a map $\delta_{\lambda}: X \to X$
which fixes $S$ and 
  satisfies $d(\delta_{\lambda} x, \delta_{\lambda} y) = \lambda d(x,y)$ for all $x, y \in X$. 
We say that the metric space  $X$ admits dilations if there is a dilation of $X$ with dilation factor $\lambda$ for each
  $\lambda > 0$.
  \end{definition} 
Spherical and hyperbolic metrics admit no dilations.  K3 fails for both.
  
 
 \section{Keplerian symmetries.}

 We  will restrict ourselves to   metric spaces  $X$ which      
\begin{align}
  &\bullet \text{are homogeneous} \\
 &\bullet \text{are isotropic} \\
 &\bullet \text{admit dilations.} 
\end{align}
We will call these three properties the {\it Keplerian symmetry assumptions}.
 
 {\it Historical Motivation.} 
 Newton's biggest victory was probably his derivation of   Kepler's laws K1,K2, K3,  from 
 more basic  laws:  
 Galilean invariance,  his equation $F = m a$, and the specific choice of   force $F$ as `$1/r^2$.'
 From these laws  he   derived what we today    call Kepler's differential equation $\ddot q = - q /|q|^3$
 and thence K1-3.  A subset of the Galilean group is the group of
 spatial isometries and this relates to   homogeneity and isotropicness.
 Dilations, as discussed above, are included so as to get a version of Kepler's third law.

We recall   the formal definition homogeneity and isotropicness.   Let  $Isom(X)$ denotes the group of isometries of $X$.
 Homogeneity  asserts that $Isom(X)$ acts transitively on $X$.   Isotropicness   asserts that $Isom(X)$
acts transitively on the space of directions through any point $S \in X$.  
The sphere and the hyperbolic plane  are homogeneous and isotropic, but they do not admit dilations.
\begin{proposition} If a \Ri manifold is homogeneous and admits dilations then it is a Euclidean space.
\end{proposition}
 Proof [sketch].  See Gromov \cite{Gromov1}, prop. 3.15.  Gromov defines the metric tangent
cone $T_p X$ of any metric space at any point $p \in X$  as the pointed limit $(X, \lambda d)$ as $\lambda \to \infty$.
This limit need not always exist, but it does exist for \Ri manifolds and equals the usual
tangent plane, with its induced Euclidean metric.  If the metric $d$  admits a  dilation  with scale factor  $\lambda$
then $(X,d)$ is isometric to $(X, \lambda d)$.  Letting $\lambda \to \infty$ we see that such an   $X$ is isometric to its metric tangent cone $T_p X$ for all $p$.  QED

Consequently, if we insist on satisfying all three Keplerian symmetries  (1)-(3) while 
also leaving the realm of Euclidean spaces,  we must 
also leave the world of Riemannian manifolds!
The simplest   non-Euclidean  metric space satisfying (1)-(3) is the   Heisenberg group
with its \sR metric.

 \section{Kepler's problem and the Laplacian}
 
 Before formulating the Kepler-Heisenberg problem, we 
 look into how the standard Kepler problem fits within the framework of
 ``natural mechanical systems'' and thus how it generalizes to general \Ri manifolds.
 This background will yield a straightforward way to place the Kepler problem in the Heisenberg context.
 
The Hamiltonian for the  standard Kepler problem on $\R^3$ is
$$H = \fr{1}{2} (p_x ^2 +p_y ^2 + p_z ^2) - \fr{\alpha}{r},$$ where $r = \sqrt{x^2 + y^2  + z^2}$ and $\alpha>0.$
Why the $1/r$ potential?  Perhaps the best answer is that
$U = \fr{1}{4 \pi r}$ is the fundamental solution for  the Laplacian $\Delta$ on (Euclidean!) $\R^3$, 
i.e. the solution to $\Delta U = -\delta_0$.  (See \cite{Albouy} and references therein.)  The choice of sign convention is due to the positivity of the operator $-\Delta$.

The kinetic term in $H$ is the principal symbol of the Laplacian,
so we can write
\begin{equation*}\tag{*}
H (q, p) = \fr{1}{2} \sigma_{\Delta} - \alpha \Delta ^{-1}_q
\end{equation*}
where $\sigma_L$ denotes the principal symbol of $L$, and 
$\Delta^{-1} (q) = K(q,0),$ where $K(x,y) = U(x-y)$ is  the Green's function for the Laplacian
(and where $\alpha = \fr{1}{4 \pi}$).  This reformulation suggests
that we can pose   `Kepler problem' as a Hamiltonian system on any `space' $X$ with a `Laplacian' $\Delta$.

This prescription (*)  for $H$  leaves us with a number of puzzles. 

{\bf Problems.}  What is the cotangent bundle of an arbitrary `space' $X$?
Assuming we make sense of $H$ as a function on the cotangent bundle of
$X$, then what are Hamilton's equations on $T^*X$?   Can we ever compute the  fundamental solution 
$\Delta^{-1}_q$ of our Laplacian?  

All these questions have  answers in the \Ri case. The principal symbol
has the coordinate expression $$\sigma_{\Delta}(p) = \Sigma g^{ij} (q) p_i p_j$$ -- it is the standard
cometric of kinetic energy.  The fundamental solution  of the Laplacian has been  explicitly computed
for hyperbolic $n$-space, so we have a   hyperbolic Kepler problem.   

If $X$ is a compact manifold without boundary,
then the fundamental solution $\Delta_q ^{-1}$
does not exist for topological reasons. For example, we cannot have   a single gravitational  source on the sphere.
There must be an opposing sink elsewhere on the sphere.  
To formulate the Kepler problem on the sphere, one places the sink  antipodally to the source.
See \cite{Diacu1} or \cite{Serret} for a precise formulation.

\section{ Kepler's Problem on the Heisenberg Group!} 

\subsection{Heisenberg Geometry}
 Consider $\R^3$ with standard $x,y,z$ coordinates,   endowed with the two vector fields
$$X = \dd{}{x} -\fr{1}{2} y \dd{}{z} \qquad Y  = \dd{}{y} +\fr{1}{2} x \dd{}{z}.$$
Then $X, Y$ span the canonical contact distribution $D$ on $\R^3$ with induced Lebesgue volume form.  Curves are called horizontal if they are tangent to $D$.
Declaring $X, Y$ \on defines the standard \sR structure on 
the Heisenberg group and yields the Carnot-Carath\'{e}odory metric $ds^2_{\mathbb H}=dx^2+dy^2$.  Geodesics are qualitatively helices: lifts of circles and lines in the $xy$-plane.  The horizontal constraint implies that the $z$-coordinate of a curve grows like the area traced out by the projection of the curve to the $xy$-plane. See Figure \ref{geod} and Chapter 1 of \cite{Tour}.

\begin{figure}
\centering
\includegraphics[width=.3\textwidth]{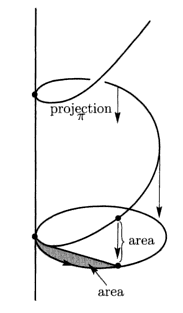}
\caption{A Heisenberg geodesic}
\label{geod}
\end{figure}

The Heisenberg (sub)Laplacian is   
$$\Delta = X^2 + Y^2, $$
a second order subelliptic operator, and the only correct choice for `Laplacian' on the Heisenberg group.  We have
$$[X, Y] =  \dd{}{z}=:Z $$
and $[X, Z] = [Y, Z] = 0$.  There are the commutation relations of the Heisenberg Lie algebra,
hence the name.  The Heisenberg group $\H$ is the simply connected Lie group
with Lie algebra the Heisenberg algebra and is diffeomorphic to  $\R^3$.
In $x,y,z$ coordinates the Heisenberg group law reads
\[(x_1,y_1,z_1)\cdot(x_2,y_2,z_2)=(x_1+x_2, y_1+y_2, z_1+z_2+\tfrac{1}{2}(x_1y_2-x_2y_1)).\]
Left multiplication is an isometry and the vector fields $X,Y$ are left invariant.

\subsection{The Heisenberg Kepler Problem}\label{KepHeis}

Folland (\cite{Folland}) has derived an  explicit  formula for the fundamental solution
for the Heisenberg Laplacian!  It is
$$U : = \Delta_q ^{-1}  =    {\alpha \over \rho^2} , \qquad  \rho = \{ (x^2 + y^2)^2 + \tfrac{1}{16} z^2 \}^{1/4}.$$
Here $\alpha=2/\pi.$
Let  $p_x, p_y, p_z$ be  the dual momenta to $x,y,z$  so  that
together $x,y,z, p_x, p_y, p_z$    form canonical coordinates on $T^* \mathbb H$.   
Then
\[P_X=p_x-\tfrac{1}{2}yp_z, \quad P_Y=p_y+\tfrac{1}{2}xp_z\]
are dual momenta to $X,Y,$
and \[K=\tfrac{1}{2}(P_X^2+P_Y^2)= \tfrac{1}{2} \sigma_{\Delta} \]
is the Heisenberg kinetic energy, given canonically by the cometric. (See Chapter 1 of \cite{Tour}.)
$K$  generates the \sR geodesic flow on the Heisenberg group.
We see that Keplerian dynamics on the Heisenberg group
are the Hamiltonian dynamics for the canonical  Hamiltonian
$$H =K -U.$$ 

There  is {\it no} explicit formula
  for the Heisenberg \sR distance function $||(x,y,z)||_{\mathbb H} :=d_{sr} ((x,y,z), (0,0,0)),$ measuring the distance from a point to the origin .
  So the mix of $K$ and $U$ -- of geodesic and subLaplacian -- is quite interesting 
  and it is rather remarkable that we can write down the Hamiltonian in closed
  form. 
  
  The dilation on the Heisenberg group is
$$\delta_{\lambda} (x,y, z) = (\lambda x, \lambda y, \lambda ^2 z).$$ 
Like the \sR distance, the function $\rho$ is positive homogeneous of degree $1$ with respect to this dilation.
Since the Heisenberg sphere is homeomorphic to the Euclidean sphere,
the standard argument which shows that any two norms on $\R^n$ are Lipshitz equivalent
shows that $\rho$ and $||\cdot||_{\mathbb H}$ are Lipshitz equivalent: there exist positive constants $c, C$
such that $c \rho(x,y,z) < ||(x,y,z)||_{\mathbb H} < C \rho(x,y,z)$ for $(x,y,z) \ne 0$. 

Following the procedure described in Section \ref{thirdlaw}, we find that if a curve $\gamma$ solves Newton's equation $\ddot \gamma(t) =\nabla U(\gamma(t))$, where $\nabla$ denotes the \sR gradient, then so does 
\[ \gamma_{\lambda}(t) := \delta_{\lambda}(\gamma(\lambda^{-2}t)).
\]
Then given a periodic orbit $\gamma$ with period $T$ (see Section \ref{periodic}), we get a family of periodic orbits $\gamma_{\lambda}$ with periods $\lambda^2 T$.  Choosing a suitable notion of the `size' $a$ of a periodic orbit yields the Heisenberg version of Kepler's third law:
\[T^2=Ca^4.
\]

The isometry group of the Heisenberg group  is generated by  translations and rotations.
The translations denote   the   action of the
Heisenberg group on itself by  left multiplication. 
These project to translations of the $xy$-plane.  The rotations form  the circle   group of 
rotations about about the $z$ axis. In addition we have the discrete  
`reflection' $(x,y,z) \mapsto (x, -y, -z)$. 
Translations  act transitively: the Heisenberg group is homogeneous.
Rotations   act transitively on (allowable) directions: the Heisenberg group is isotropic.  
 Thus the Heisenberg group enjoys the three Keplerian symmetry properties.

 \subsection{Hamiltonian Dynamics}\label{HamDyn}
 The dilation on phase space $T^*\mathbb H$ is
 \[ \delta_\lambda \colon (x, y, z, p_x, p_y, p_z) \mapsto  (\lambda x, \lambda y, \lambda^2 z, \lambda^{-1} p_x, \lambda^{-1} p_y, \lambda^{-2} p_z).\]  This is generated by the function $J=xp_x+yp_y+2zp_z$, which satisfies $\dot J=2H$.  When $H=0$, $J$ is a first integral.  Note that $\delta_{\lambda} \colon H \mapsto \lambda^{-2}H.$

 Now change to cylindrical coordinates $(r, \theta, z)$ on $\mathbb H$.  We have the induced conjugate momenta $p_r =(xp_x+yp_y)/r$ and $ p_{\theta}=xp_y-yp_x$.  Our Hamiltonian is
 \[ H=\tfrac{1}{2}p_r^2+\tfrac{1}{2}\Big(\frac{p_{\theta}}{r} +\tfrac{1}{2}rp_z\Big)^2 -\alpha\rho^{-2}. \]
 Note that this does not depend on $\theta$ due to rotational symmetry, and the corresponding angular momentum $p_{\theta}$ is conserved.

On the smooth submanifold of phase space $\{H=0\}$, we have three (independent) conserved quantities $H, p_{\theta}$, and $J$, and a theorem of Arnold (see \cite{Arnold}) says that our system is integrable by quadratures here.  See Figure \ref{orbits} for approximations of orbits which exhibit this integrable behavior as well as the helical Heisenberg geometry.  For this reason, we will mostly focus on the $H=0$ case.  This is especially justified in light of the following.

\begin{lemma}
Periodic orbits must have zero energy.
\end{lemma}

 \begin{proof}
 If $\gamma (t)=(x(t), y(t), z(t), p_x(t), p_y(t), p_z(t))$ satisfies $\gamma(0)=\gamma(T)$ for some $t=T$, then $J=xp_x+yp_y+2zp_z$ is also periodic.  But we know the time derivative of $J$ is constant, given by $\dot J=2H$.  Since $J$ cannot be monotonically increasing nor decreasing, we must have $\dot J=2H=0$, so $H=0$.
 \end{proof}

\begin{figure}
\centering
\begin{tabular}{cc}
\includegraphics[width=.5\textwidth]{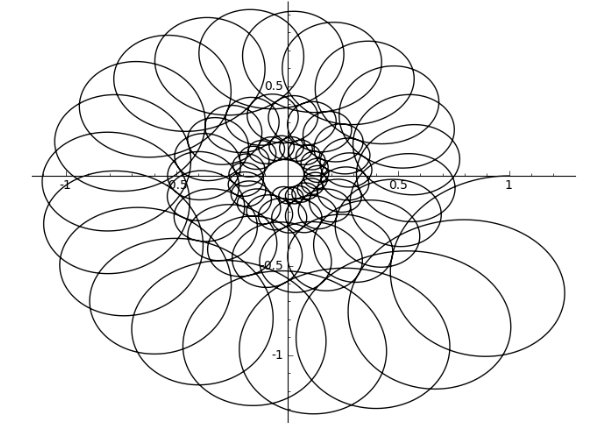}
\includegraphics[width=.5\textwidth]{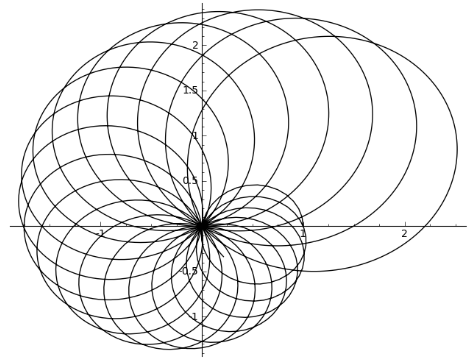}
\end{tabular}
\caption{Projections of zero-energy orbits to the $xy$-plane.}
\label{orbits}
\end{figure}

Periodic orbits exist and the existence proof forms part of C.S.'s thesis -- see Section \ref{periodic} below.  We will momentarily report progress with integration of the $H=0$ system, but first we gather other dynamical results.

\begin{proposition}
If $H<0$ then any solution is bounded.
\end{proposition}

\begin{proof}
Suppose $H=-h$ where $h$ is positive.  Then $K-U=-h$, so 
\[ U=K+h\ge h,\]
since $K$ is always non-negative.  Then a solution $(x(t),y(t), z(t))$ in configuration space must satisfy
\[ 0\leq((x^2+y^2)^2+\tfrac{1}{16}z)^{1/2} < \frac{\alpha}{h}, \]
where $\alpha$ and $h$ are positive constants.
\end{proof}

 \begin{proposition}
 The only solutions in the plane $z=0$ are lines through the origin.
 \end{proposition}
 
 \begin{proof}
 The equations for $\dot z$ and $\dot \theta$ satisfy the relation
 \[\dot z= \tfrac{1}{2}r^2\dot \theta. \]
 For a path lying in the plane $z=0$, this implies either $r=0$ or $\dot \theta=0$.  In the first case, the path is trivial.  In the second, it lies on a line through the origin. Such a curve may be parametrized  by
 \[ \gamma(t)=(c_1t^{1/2}, c_2t^{1/2}, 0, \tfrac{1}{2}c_1t^{-1/2}, \tfrac{1}{2}c_2t^{-1/2}, 0).\]
 It is easy to verify that the desired equations are satisfied, and that $H=0$.
 \end{proof}
 
 \begin{proposition}
 The only solutions constant in configuration space are
 \[ \gamma(t)=(0,0,k,0,0,-\tfrac{4\alpha}{k^2}t). \]
 \end{proposition}
 
 \begin{proof}
 This is an easy calculation.  Note that such solutions are unbounded in phase space, and satisfy $H<0$.
 \end{proof}

 Next, we explicitly integrate the equations of motion on a codimension 3 submanifold, and recover conics reminiscent of the Euclidean Kepler problem.  Consider the smooth submanifold $N=\{z=p_z=p_{\theta}=0\}$.  This submanifold is invariant under the dynamics, since $\dot z= \dot p_z= \dot p_{\theta}=0$ on $N$.  The Hamiltonian is
\begin{align*}
 H|_N
&= \tfrac{1}{2}p_r^2-\frac{\alpha}{r^2},
\end{align*}
which has the form of a classical central force problem in the plane.  Fix an energy level $H|_N=h$.  Then since $p_r=\dot r$, we can explicitly solve for $r(t)$ as follows.  

\begin{proposition}
On $N$, $r(t)$ traces out a hyperbola if $h>0$, an ellipse if $h<0$, and a parabola if $h=0$.
\end{proposition}

\begin{proof}
The Hamiltonian may be rewritten as the simple ODE
\[\tfrac{1}{2} \Big(\frac{dr}{dt}\Big)^2=\frac{\alpha}{r^2}+h.\]
Assume temporarily that $h \neq 0$.  Integrating, we find
\begin{align*}
t&= \int dt 
=\tfrac{1}{\sqrt 2}\int \frac{r}{\sqrt{\alpha+r^2h}}dr
=\tfrac{1}{h\sqrt 2}\sqrt{\alpha+r^2h},
\end{align*}
which may be rewritten
$r^2-2ht^2=-\tfrac{\alpha}{h}.$
Since $\alpha>0$, this curve in the $t,r$-plane is an ellipse for $h<0$ and a hyperbola for $h>0$.

If $h=0$, we find that 
\[t=\tfrac{1}{\sqrt{2\alpha}}\int rdr= \tfrac{1}{2\sqrt{2\alpha}}r^2, \]
and thus
$r^2=\sqrt{8\alpha} \ t.$
\end{proof}

We conclude this section with the following conjecture:
 \begin{conjecture}
 There is an open set of initial conditions whose orbits are asymptotic to helices.
 \end{conjecture}
 
 This behavior is suggested by numerical experiment and by the fact that $U$ and its derivatives tend to zero as orbits tend towards $\infty$. 

 \subsection{Integration of the case $H=0$}
 We now focus on the $H=0$ case and reduce the integrability of the equations of motion to the parametrization of a family of degree 6 algebraic plane curves.
 
Let $\tilde H=\frac{K}{U}$.  Then integral curves for $\tilde H$ are the same as geodesics for the metric $Uds^2_{\mathbb H}$.
 When $H=0$, this is the same as the metric $(H+U)ds^2_{\mathbb H}$, whose geodesics correspond to integral curves for $H$, according to the Jacobi-Maupertuis principle.
 Thus, the flow of $H$ is the same as the flow of $\tilde H$ up to reparametrization on the hypersurface $\{H=0\}=\{\tilde H=1\}.$
 
A short calculation shows that both $J$ and $p_{\theta}$ Poisson commute with $\tilde H$. 
(Recall $\{H, J\}=2H$.)
 This demonstrates the scale invariance of $\tilde H$;
 $\delta_\lambda \colon \tilde H \mapsto \tilde H.$
 More importantly, we have three independent quantities conserved by the flow of $\tilde H$.  Thus, we have an integrable system on $\{H=0\}=\{\tilde H=1\}.$
 
 Change our third coordinate $z \mapsto v=z/r^2$.  We have conjugate momenta $\tilde p_r,  p_{\theta}, p_v$.  In these coordinates, we have $J=r\tilde p_r$ (as in the Euclidean case) and 
 $\tilde H=\tilde H(p_{\theta}, J, v, p_v).$
 On the submanifold $\{H=0\}$, we have $\tilde H=1$.  Also, the initial conditions determine the constants $J$ and $p_{\theta}$.  Thus, given initial conditions, $\tilde H$ is a function of $v$ and $p_v$ only.  We arrive at the following result.
 \begin{proposition}
 When $\tilde H=1$, any solution must project to an algebraic curve in the $v, p_v$-plane.
 \end{proposition}
 
These curves are naturally degree 10 but can be reduced to degree 6 by changing variables.  Examples are shown in Figure ~\ref{v_curves}. If we can parametrize these curves, we should be able to bootstrap up to find explicit solutions.
 \begin{figure}
 \centering
 
 \begin{tabular}{cc}
 \includegraphics[width=.5\textwidth]{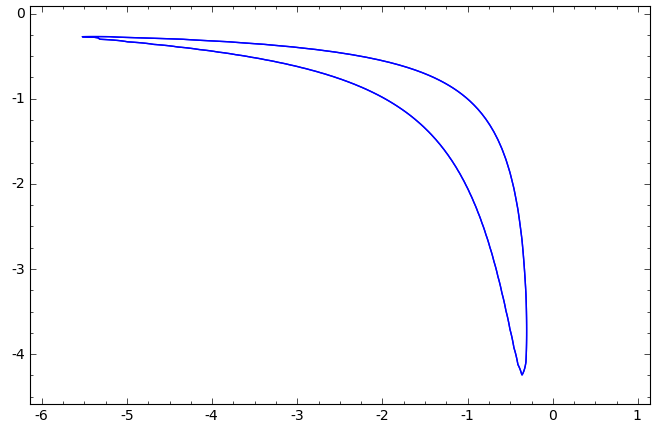}
 \includegraphics[width=.5\textwidth]{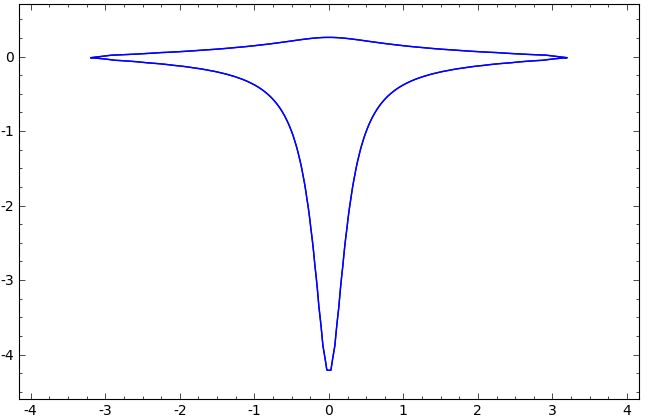}
 \end{tabular}
 \caption{Curves in the $v, p_v$-plane corresponding to $J=3,\ p_{\theta}=1$ (left) and $J=0,\ p_{\theta}=1$ (right).}
 \label{v_curves}
 \end{figure}

 \subsection{Periodic Orbits}\label{periodic}

Despite the fact that the $H=0$ case is integrable, we have not been able to explicitly solve the equations.  However, we know that periodic orbits exist.
 
Take $L=K+U$ as our Lagrangian and impose the horizontal constraint $\dot z=\tfrac{1}{2}x\dot y - \tfrac{1}{2}y \dot x$.
 Then any trajectory $\gamma$ must lie on the zero set of the function  
 \[G=\tfrac{1}{2}x\dot y - \tfrac{1}{2}y \dot x-\dot z.\]
 The calculus of variations tells us that if $\gamma \colon [0,T] \to \mathbb H$ is a minima of the action functional $\int_0^T  Ldt$ which also satisfies our constraint, then there exists a scalar $\lambda=\lambda(t)$ such that $\gamma$ is a minima of the modified action functional 
 \[ A(\gamma) = \int_0^T  L_{\lambda}(t, \gamma, \dot \gamma) dt,\]
 where we have written $L_{\lambda}(t, \gamma, \dot \gamma)= L(t, \gamma, \dot \gamma)-\lambda(t)G(t, \gamma, \dot \gamma)$.
 Setting the first variation of $A$ equal to zero and integrating by parts yields the Euler-Lagrange equations:
 \begin{align*}
 \ddot x&=-\lambda\dot y - \tfrac{1}{2} \dot \lambda y - 2\alpha x(x^2+y^2)\rho^{-6} \\
 \ddot y&=\lambda\dot x + \tfrac{1}{2} \dot \lambda x - 2\alpha y(x^2+y^2)\rho^{-6} \\
 \dot \lambda &= -\tfrac{\alpha}{16}z\rho^{-6}.
 \end{align*}
 When $\lambda=p_z$ we find that these agree with Hamilton's equations.
 
 We are confident that the direct method in the calculus of variations applied to $A(\gamma)$ will yield a proof of the existence of periodic orbits. One works in the Hilbert space $H^1(S^1, \mathbb H)$ and requires that admissible curves are horizontal and satisfy the symmetry conditions
  \begin{equation*} \tag{S1}
  \gamma(t+ T/3)=R_{2\pi/3}  \gamma(t)
 \end{equation*}
 and
  \begin{equation*} \tag{S2}
  z(t+ T/2)=-z(t),
 \end{equation*}
  where \[R_{2\pi/3}=\begin{bmatrix} -\frac{1}{2} & -\frac{\sqrt 3}{2} & 0\\ \frac{\sqrt 3}{2} & - \frac{1}{2} & 0\\0 & 0 & 1 \end{bmatrix}.  \]
 Any admissible curve is therefore necessarily periodic, with additional symmetry.  A suggestive approximation of such a curve is shown in Figure \ref{orbit1}.
 
 The idea is to choose a minimizing sequence $\gamma_n$ of curves in this space, and show that they converge within the space to some $\gamma_*$.  Applying elementary analysis and the principle of symmetric criticality shows that $\gamma_*$ must minimize the action, thereby satisfying the Euler-Lagrange equations.  A central difficulty lies in proving that $\gamma_*$ does not pass through the singularity at the origin.  A full existence proof is expected to appear in the thesis of C.S.

\begin{figure}
 \centering
 \begin{tabular}{cc}
\includegraphics[width=.6\textwidth]{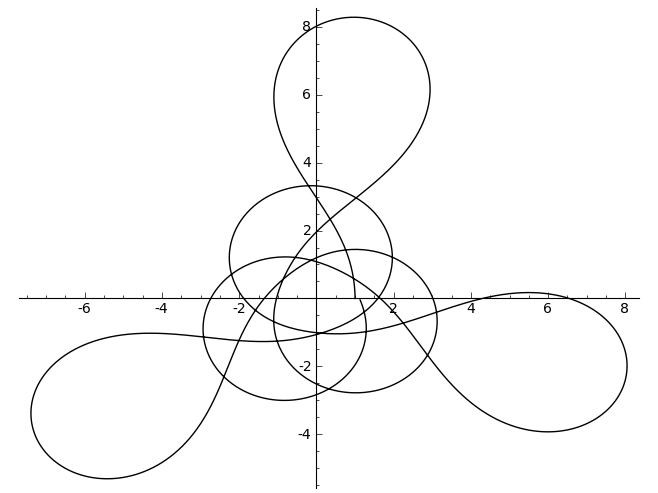}
\raisebox{3.4mm}{
\includegraphics[width=.4\textwidth]{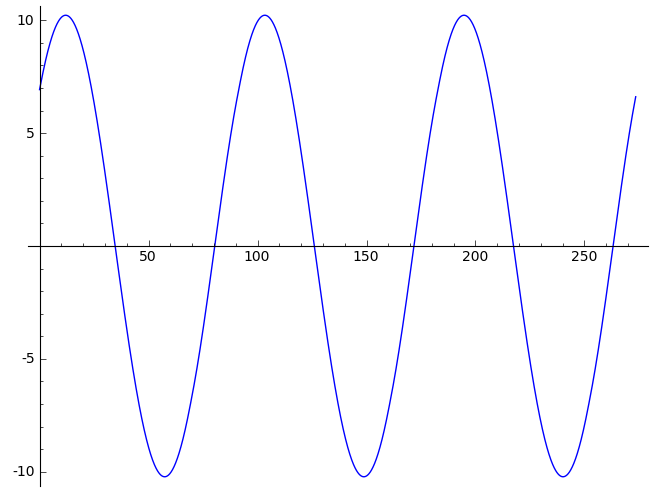}}
 \end{tabular}
 \caption{Projection of an orbit to the $xy$-plane (left) and $z$-coordinate over time (right)}
 \label{orbit1}
 \end{figure}

 \subsection{A failure of reduction.}

\vskip .2cm

Newton reduced his   two-body problem
  to the  Kepler problem in Euclidean space.
There is  no  analogous reduction
for the two-body problem on the Heisenberg
group, nor is there for the two-body problem on
the sphere or in hyperbolic space.
 We  discuss the geometric roots of this failure.

We  begin by  writing  down
the Heisenberg two-body problem.  Let   $q_1 , q_2 \in \H \cong \R^3$ denote the
positions
of   two bodies
moving in the Heisenberg group $\H$. Let their  masses be  $m_1, m_2$.
Their individual kinetic energies are
$$K_i = \fr{1}{2 m_i} ((P^{(i)}_X)^2 +  (P^{(i)}_Y )^2)$$
where   $P^{i}_X, P^i _Y$ are the horizontal momenta
of each body, as in Section \ref{KepHeis}.
The Heisenberg two-body problem is defined by the
Hamiltonian
$$H = K_1 + K_2 -  \kappa m_1 m_2 U(q_1 ^{-1} q_2),$$
where $\kappa$ is the Gravitational constant
and $U$ is Folland's fundamental solution.
  $H$  is a Hamiltonian on the cotangent bundle of $\H \times \H$,
and is invariant under the (cotangent lift of the  left) translation $(q_1, q_1) \mapsto (g q_1, g q_2)$,
$g \in \H$.

We know of two  derivations of   Kepler's   problem (on Euclidean space)
from Newton's two-body problem.  We will call these the
`algebraic' and the `group-theoretic' derivations. The  `algebraic derivation'  begins with
the equation  $F=m a $
for each body.  Divide the  equation for each body  by its mass   to get
equation for the acceleration  $\ddot q_i$ of each body's position vector $q_i$.  Subtract
one   equation from the other to obtain the    ODE of   Kepler's problem,   $\ddot q = - \alpha q /|q|^3,$
for  the difference
vector   $q= q_1 - q_2$.
The `group theoretic  derivation' depends  on the conservation of the
  total  linear momentum,   the   invariance of  Newton's   mechanics
with respect to    Galilean boosts, and the abelian nature of the translation group.
If $P$ is the total linear momentum and $M$ the total mass, we
boost by the velocity  $-P/M$ to get to a new representation of the same dynamics
in which the total linear momentum is zero.   Then we reduce by translation  at the value  $0$
by placing the center of mass at the origin.   Finally, we compute that each mass separately
satisfies Kepler's equation with the origin -- the center of mass --  now playing the role of   ``sun''.

The algebraic derivation fails on the Heisenberg group because the `difference vector' $g_1 (t)^{-1} g_2 (t)$ of two  Heisenberg
geodesics is not a Heisenberg geodesic.   Why is this lack of being a geodesic a problem?
Set  the Heisenberg Gravitational constant $\kappa = 0$
so the two-body problem reduces to two uncoupled Heisenberg geodesic problems.
Play the algebraic game. Our  `difference vector'  does not satisfy the Heisenberg geodesic equations
or any other pretty Hamiltonian equation.  But in the Newtonian-Euclidean case, the difference vector
travels like a free particle, i.e., moves in a straight line -- as it should with $\alpha = 0$ in Kepler's problem.
Things will just get worse for $\kappa \ne 0$.

The failure of the group theoretic derivation goes a bit deeper and is  perhaps more enlightening.
What is a `Galilean boost' for  an arbitrary Lie group? We choose some `translation velocity'
$\xi$ and multiply elements  $x_0$  by $exp(t \xi)$.  Euclidean space enjoys the wonderful
property that $exp(t \xi) x_0 = x_0 + t \xi$ describes free motion; it is a geodesic.
This assertion is decidedly false for  the Heisenberg group: the orbits of
(left or right) translates of one parameter subgroups are not Heisenberg geodesics.
As a result,
applying a boost to a solution $(q_1 (t), q_2 (t))$ to
the Heisenberg two-body problem  will not yield a solution.  There is
a conserved  total `linear momentum':   the momentum map for the
(left) translation action.   But we cannot use it to  `Galilean boost' the `center of mass velocity'
 down to zero.
Even if this total linear  momentum  were initially zero,
we still seem to be stuck.   The non-Abelian nature of the group appears
to block us from writing the reduced Hamiltonian at zero as a Kepler Hamiltonian
on the `diagonal group' of elements  $q = q_1 ^{-1} q_2$.

In spherical and hyperbolic geometry, reduction of the two-body problem to
the Kepler problem fails for similar reasons. See \cite{Diacu2}.
In the spherical case, Shchepetilov \cite{Shchepetilov}  used the Morales-Ramis theory to  prove that
the two-body  problem in these two geometries is not meromorphically integrable.

{\bf Question.}  Is  the two-body problem on the Heisenberg group non-integrable?

   \section{Kepler's Problem on a Lattice.}  
   
   Lattices admit one-sided dilations: we can scale a lattice $\L$ by a  positive integer $c$ 
   and land back in the lattice, stretching all distances by $c$.
   They   admit Laplacians. So we  
   might be able to  begin to investigate    Kepler's 3rd law on $\L$.   
   
   What are Newton's equations on $\L$?  Since we must hop from lattice site to lattice site,
   we must choose our  time variable  $t$ to be discrete: 
   $$t = \ldots,  -1, 0, 1, 2, \ldots. $$
   A `solution' to Newton's equations will  then be a `discrete curve'
   $$\gamma: \Z \to {\mathbb L,}   \qquad \L \text{ our lattice,} $$
   satisfying a difference equation which mimics Newton's equations.  In
   1st order Hamiltonian form these equations should resemble
   $$\fr{d \gamma}{d [t]} = p$$
   $$\fr{d p}{d [t]}  = - \nabla V (\gamma (t)),$$
   where the differential is the discrete difference operator
   $$\fr{d \gamma}{d [t]} = \gamma (t+1) - \gamma(t),$$
   and where 
   $V: \L \to \R$ is our potential.
   The standard interpretation of $\nabla V$ is in terms of its differential
  $$dV(\ell):  \E_{\ell} \to \R, \quad \ell \in \L,$$
   where $\E = \E_{\ell}$ is the set of edges (chosen lattice generatos)
    leaving the lattice site $\ell,$
   and where
   $$dV(\ell) (e) = V(\ell^{\prime}) - V(\ell), \quad   e = [\ell, \ell^{\prime}] \text{ an edge}.$$
   Then we can rewrite our Newton difference equations as
  \begin{eqnarray}
\gamma(t+1) = \gamma(t) + p(t) \\
 p(t+1) = p(t) - dV(\gamma(t)).
\end{eqnarray}
   What is  the momentum, $p(t)$? We add to it   $dV(\gamma(t))$,
   so  it must lie  in the
   same space as $dV,$ which is
   $$\E^* = \E_{\ell} ^* = \text{ real valued functions on } \E_{\ell}.$$
    This `cotangent space at $\ell$' is 
   a vector space isomorphic to $\R^d,$ where $d$ is the degree of a vertex: the number of edges leaving
   $\ell$.   
Good.  Now, how do we add   $p(t) \in \E^*$ to a lattice site $\ell = \gamma(t) \in \L$
   in order to get a new lattice site $\gamma(t+1)$ as in the 1st Newton equation?  
    {\it We seem to be missing the `mass matrix' or `cometric' of mechanics.}
    
   \begin{definition}
   A lattice cometric is a   `non-trivial'  map
   $$\M :  \E^* \to \L.$$ 
   \end{definition}
   With this tentative definition we can now try to write down `Newton's equations'
     \begin{eqnarray}
\gamma(t+1) = \gamma(t) + \M (p(t)) \\
 p(t+1) = p(t) - dV(\gamma(t)),
\end{eqnarray}
   which define  a discrete dynamics on the phase space $\L \times \E^*.$  
    The resulting dynamics have some vague relation to the corresponding formal Hamiltonian
   $$H( \ell, p) = \fr{1}{2} p \M  p + V(\ell),$$
  but we aren't sure how to interpret the term $p\M p$.

\section{ Kepler's problem on $\Z$.}
 
The Laplacian $\Delta$ on $\Z$ is given   by $\Delta f (n) = f(n+1) - 2 f(n) + f(n-1)$.  

{\sc Exercise.}  Show that $U (n) =- \fr{1}{2} |n|$ is a fundamental solution for the Laplacian
  on $\Z$ with source at $S = 0$. 

  Take the Kepler constant  $\alpha = 2$  so  that the `Newtonian potential'  $V = -\alpha U$  is  $V = |n|.$
  The Hamiltonian is
  $$H(n, p) = \fr{1}{2} p^2 + |n|.$$
Since  $\fr{d}{d[n]} |n| = sgn(n)$ is the sign of $n$, $1$ if $n > 0$ and $-1$ if $n \leq 0$,
we find  that with this choice of $\alpha$  the discrete gradient is integer valued.  We get a good discrete
dynamical system.  
  Newton's equations (in 1st order form)  become 
$$n(j+1) = n(j) + p(j)$$
$$p(j+1) = p(j) - sgn(n).$$

A solution is depicted in Figure \ref{intpic}.
\begin{figure}
\centering
\includegraphics[width=.6\textwidth]{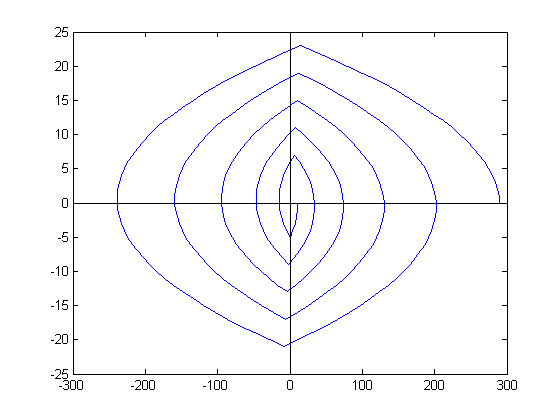}
\caption{A Kepler orbit in integer phase space}
\label{intpic}
\end{figure}

At each iteration, the `lattice momentum' $p$ decreases by one as long as  $n > 0$ and    increases by $1$
as long as $n < 0$.    (We have to make a choice at $n=0$; above we chose $sgn(0)=-1.$)
 Note that as long as the   initial condition $p(0)$ is an  integer,  it remains an integer, and we   stay on the lattice!

 Because of this happy coincidence with $p$'s evolution, we did not need to worry about where $p$ lived. It is a real number that
 happens to evolve to stay integral.  Our momentum space is  not  $\E^* \cong \R^2$.  (There are two directions, right and left,
 on the lattice, hence the dimension 2.)  We also did not need to choose a `lattice cometric'  $\M$.
 If any choice was made, it seems to have been   `$1$'  as written in the Hamiltonian.
 The happy coincidence does not happen when we go up to the rank 2 lattice.

 \subsection{ Kepler on the rank 2 lattice.}
 
 The rank 2 lattice is $\Z^2$ with elements written $\ell = (n, m)$.
 As a metric space, we use the distance
 $$d((n, m), (n^{\prime}, m^{\prime}) = |n-n^{\prime}| + |m - m^{\prime}|.$$
 The Laplacian is 
  $$\Delta f (\ell) =  \sum_{\ell^{\prime} : \ d (\ell, \ell^{\prime}) =  1} (  f(\ell^{\prime})- f(\ell)).$$
  Let $U$ denote the fundamental solution, this being
  the `most bounded' solution  to $\Delta U = \delta_{0}$, where $\delta_0$
  is the lattice delta function corresponding to placing the sun at the origin.
    (There is a lattice Liouville theorem, so $U$ can be made  unique up to an additive constant.) 
There is no closed form expression for   $U.$
  However, any particular value of $U$ can be computed recursively.
  Indeed, the fundamental solution is a well studied object with applications to
  the theory of electrical circuits (\cite{Cserti}), solid state physics, and quantum mechanics. 
 Some values of the lattice Green's function are reproduced  in Figure \ref{latticeGreen}
 from \cite{Hollos}. (We thank the brothers Hollos for permission to reproduce their table.)
  
   \begin{figure}
   \centering
   \includegraphics[width=.6\textwidth]{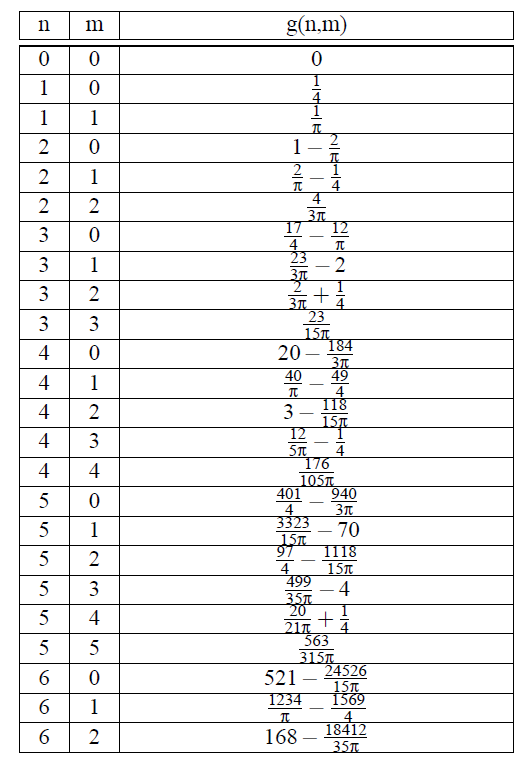}
   \caption{Values of the Green's function $g$ on $\mathbb Z^2$, taken from \cite{Hollos}}
   \label{latticeGreen}
   \end{figure}

  We write the formal Hamiltonian
  $$H(\ell, p) = \fr{1}{2} p \M p - \alpha U(\ell)$$
  and derive Hamilton's equations
  $$\ell (t+1) = \ell(t) + \M p(t)$$
  $$p(t+1) = p(t) + \alpha dU(\ell(t))$$
  where
  $$\M:  \R^4 = \E^* \to \Z^2.$$
  
   The vector $p$ is a 4-vector with components
  $(p_{up}, p_{down}, p_{right}, p_{left})$
  corresponding to the 4 edges, which are the 4 directions of motion,
  through each vertex.  We have, for example
  $dU(\ell)_{up} = U( \ell + e_2) - U(\ell)$
  where $e_2 = (0,1)$ represents motion in the `up' direction.
  The second Hamilton equation makes sense. 
  
  When we try to parse the first Hamilton equation
  we get   stuck.
  What do we take for $\M: \R^4 \to \Z^2$?
  We require $\M$  to be non-constant.
   Certainly $\M$ will not be  continuous! 
  Ideally $\M$ is `linear':
  $$\M (kp) = k \M (p),\quad k \in \Z,$$
  but this is probably not possible in any reasonable sense.   
  One possibility for $\M$ is to  argue that there is a  `canonical' projection
  $\Pi: \R^4 \to \R^2$, for example, 
  $\Pi (p_{up}, p_{down}, p_{right}, p_{left}) = \fr{1}{2}(p_{up} + p_{down}, p_{right} + p_{left})$,
  and a canonical  embedding of our lattice as  $\Z^2 \subset \R^2$.
  Then choose   
  $\M (p)$ to be the lattice point closest to $\Pi(p)$.
  This leaves us to worry about what to do if $\Pi(p)$ is midway between lattice points.
  Flip a coin?
  
  We are stuck and look forward to some of our readers
  unsticking us.

\subsubsection{Euler-Lagrange formulation.}  

We can make a bit more sense of the Euler Lagrange version of
lattice dynamics.   Fix a  positive integer $T$, the `time of flight,'
and initial and final vertices, $v_0, v_1 \in \Z^2$. 
There will be two formulations.
In both, we consider   discrete paths $\gamma: \{0, 1, \ldots, T \} \to \Z^2$ 
which join $v_0$ to $v_1$ in time $T$,  and we minimize an `action functional'
$A$ among all such discrete paths.

Version 1:  Minimize the action 
$$A (\gamma) = \sum_{t=0} ^T  \{ \fr{1}{2} (| \fr{d}{d[t]} \gamma (t)  |_1 )^2 + \alpha U(\gamma(t)) \}$$
among all discrete paths $\gamma$  joining $v_0$ to $v_1$ in discrete time $T$.

Here $| \fr{d}{d[t]} \gamma (t)  |_1 = d(\gamma(t+1), \gamma(t)),$ so half of  its square represents kinetic energy. 

Version 2:  Call a discrete path `continuous' if either $d(\gamma(t+1) , \gamma(t)) = 1$
or $\gamma(t+1) = \gamma(t)$.  Minimize the same action as Version 1,
but now over all continuous paths.   (In this case the kinetic term  $ \fr{1}{2} (| \fr{d}{d[t]} \gamma (t)  |_1 )^2$
is either $1/2$ or $0$ at each time step.)

We are guaranteed a solution   to Version 2   exists
since there are   only a finite number of  `continuous' paths joining $v_0$ to $v_1$.
We suspect that if we move too fast the kinetic energy becomes too large, so that
Version 1 is `coercive' and one can argue that again there are only a finite number
of paths that matter.

It seems   doubtful that any decent Euler-Lagrange type 
 difference equation  `dynamics' will result from either principle.  Indeed, take the case  $\alpha = 0$
  of a   `free particle'
on the latttice, and take $v_0 = (0,0),\  v_1 = (n,m)$,  $n > m \ge 0$.   There are $(n+1)m$ shortest paths from $v_0$ to $v_1$.
Just draw box-paths, always moving either right or up. 
Their lengths are all   $ n + m = d(v_0, v_1)$. 
If $T< d(v_0, v_1)$ then there are no paths connecting the two points.
If $T = d(v_0, v_1)$
then their  actions are all $\fr{1}{2} T$.
If $T > d(v_0, v_1)$ the action remains the same;  we just stay still for the requisite times
$T- d(v_0, v_1)$.   This means either (i) all points $v_1 = (n,m)$  with $nm \ne 0$ are conjugate to $v_0$,
or (ii)  that there is no good `free' dynamical equation, so likely no good Euler-Lagrange equations
in general.

   \section{ Quantum Mechanics to Classical Mechanics on Cayley Graphs? }

By a graph here we mean the usual combinatorial collection of vertices and
edges.  We write   $\Gamma$  for the set of vertices and view $\Gamma$ as `configuration space.'
The graph Laplacian is  the operator
      $\Delta:  \ell_2 (\Gamma) \to \ell_2 (\Gamma)$
      defined by 
       $$\Delta f (v) =  \sum_{v^{\prime}: [v, v^{\prime} ] \text{ an edge}}   (f(v') - f(v)).$$
       If $\Gamma$ is finite there will be no fundamental solution; that is, there is no solution
       to $\Delta U_S = \delta_S$ where 
  where $\delta_S $ is the discrete $\delta$ function
  centered at the sun:  $\delta_S (S) = 1, \delta_S (v) =  0, v \ne S$.
  If $\Gamma$ is finite, a necessary condition for the  solvability
  of $\Delta V = f$ is  $\Sigma f(v) = 0$, which will fail for $f=\delta_S$.
  
   Regardless of whether or not
   $\Gamma$ has a Green's function, it has plenty of 
   potentials, meaning functions   $V \in \ell_2 (\Gamma)$.  Consequently for  each choice of Planck's constant $\hbar$
  we have a Schrodinger operator: 
  $$\hbar ^2\Delta + V  : \ell_2 (\Gamma) \to \ell_2 (\Gamma).$$
  There is a large active field of graph Laplacians and quantum mechanics on graphs. 
There is undoubtedly a theory of  quantum mechanics on $\Gamma$.
  
 {\bf  Challenge.  } Don't you think this
quantum mechanics ought to have a   classical limit?  If `yes' then please answer: what  are 
the correct   Newton's equations for an arbitrary potential, on an arbitrary graph? 
 
 \subsubsection{Cayley graph of a group}   Let $\Gamma$ be a finitely generated group and  $e_1, \ldots, e_d \in \Gamma$ be
 a fixed  set of generators for $\Gamma$ (so every element of $\Gamma$
 is a product of the $e_i$'s or their inverses).  Form the  graph whose  vertices are the elements 
  $x \in \Gamma$ and for which  two vertices  $x, y \in \Gamma$ are  joined  by an edge if and only if
 either $y = x e_i$ or $x =  ye_i$ for some generator $e_i$.  
 Count each edge as having length $1$. Define the distance between points
 $x$ and $y$ in $\Gamma$  to be the minimum of the lengths of the paths joining $x$ to $y$.
 This distance is always an integer, since the length of a path is just the number of edges it contains.
 
 In this representation, the `Lie algebra' of the Cayley graph will be the tangent space at the identity:
  the disjoint union of $d$ copies of $\Z$.  Alternatively,  it is the subset of $\Z^d$
 consisting of vectors for which all but one component is zero. 

{\it Example: Lattices.}  Take $\Gamma = \Z ^2$ to be the lattice of integers in the plane, with standard  generators $e_1 = (1,0), e_2 = (0,1)$.  Then the Cayley graph of    $\Z^2$ realized as above  
has the vertices of a standard infinite sheet of graph paper in $\R^2$.   Its Lie algebra
consists of integer points on the $x$-axis unioned with the collection of integer points on the $y$-axis. 

  \subsubsection{Kepler symmetries of Cayley graphs.}
 Every  Cayley graph satisfies Keplerian symmetry property  (1) of being homogeneous since $\Gamma$  acts on itself on the right  by isometries.
 View  the generators as the `directions.'  Then if  the  automorphism group  of the group $\Gamma$ acts transitively on its generating set 
 $e_1, \ldots, e_d,$ the metric is isotropic; it satisfies 
 Keplerian symmetry property  (2).  Finally we can send $e_i$ to $e_i^k$.
 In some instances this defines a group homomorphism of $\Gamma$ {\it into } itself.
 Then the Cayley graph admits one-sided dilations and so satisfies (3). 
 The   examples we know of groups whose Cayley graphs satisfy (1), (2) and (3) are the lattices $\Z^d$, the
 lattices in nilpotent groups,  and the free group on $d$
 generators.  In the continuous case, we know how to derive a Kepler's third law from the Keplerian symmetry (3).  Is there an analogous construction in the discrete case?
 
   \subsubsection{Full Disclosure -- R.M. }
  
   I have little
  interest in any kind of graph for its own sake. 
  I am not a combinatorist, nor a discrete group theorist! 
  
 In contrast to the dozens of books that Jerry wrote in his life,   I have mustered  the courage and stamina
    to write a single  book in this life.
 (Jerry continues to amaze.) In that book I devoted a chapter to trying to understand  one of the big
 ideas of Gromov  in his paper   `On Groups of Polynomial Growth \dots' (\cite{GromovGroup1}) in which he    used  \sR ideas  
 to solve  a problem in  discrete group theory.  Consider a discrete finitely generated group
  $\Gamma$. Select some generators and form the group's Cayley graph. We say the group is   `of polynomial growth'
  if the number of vertices of the Cayley graph lying inside a ball of radius $R$ is bounded by a   polynomial in $R$ as $R \to \infty$.
  (If $\Gamma$  is of polynomial growth with respect to one set of generators, it is of polynomial growth with respect to
  any other set of generators.)
  The lattices, and the integer lattice in the Heisenberg group are examples of groups of polynomial growth.
  More generally, the lattices in any Carnot group are of polynomial growth. 
  The free group on 2 generators in not of polynomial growth: its balls have exponential growth, roughly $3^R$.  
  There is a notion of a group being `virtually nilpotent,' and it was known that virtually nilpotent implies polynomial growth.
  Gromov proved the converse:  polynomial growth implies virtually nilpotent.

  Gromov's paper is mind-blowing --    the most astounding application of \sR geometry
that I know of   made by a human.  (Cats and micro-organisms have made their own astounding applications.)
Gromov scales the edges of the Cayley graph by $\epsilon$, then  takes the limit as $\epsilon \to 0$. 
He  proves, in essence,  that the result converges to a Carnot group -- a metric of \sR type on 
a nilpotent Lie group -- and from this the theorem easily follows.
 (I am  stretching the truth  here, but that is the spirit of Gromov's paper.  There are many technicalities.)
What I  find so compelling  about Gromov's paper   is the going back and forth between the  wonderful world of smooth
 metric spaces -- Lie groups even -- 
which I know and love, and  the   chopped up world of discrete objects that I find so frightening at times.
Can we similarly go back and forth in dynamics?  That is what I would like to see in
some `Kepler problem on a lattice.'

  \newpage
\bibliographystyle{amsplain}

  \end{document}